\documentclass[11pt]{amsart}
\usepackage{verbatim, latexsym, amssymb, amsmath,color,enumitem}
\usepackage{graphicx} 
\usepackage{hyperref} 
\usepackage{epsfig}
\usepackage{bm}
\usepackage{mathrsfs}
\usepackage[margin=1.25in]{geometry}
\usepackage{subcaption}

\newtheorem{theorem}{Theorem}[section]

\newtheorem{claim}[]{Claim}

\newtheorem{proposition}[theorem]{Proposition}

\theoremstyle{definition}
\newtheorem{definition}[theorem]{Definition}
\newtheorem{conjecture}[theorem]{Conjecture}

\newtheorem{remark}[theorem]{Remark}

\numberwithin{equation}{section}

\newcommand{\dist}{\operatorname{dist}}

\newcommand{\Clos}{\operatorname{Clos}}
\newcommand{\interior}{\operatorname{int}}

\newcommand{\R}{\mathbb{R}}

\newcommand{\bd}{\partial}
\newcommand{\A}{\mathcal{A}}
\newcommand{\C}{\mathcal{C}}
\newcommand{\I}{\mathcal{I}}

\newcommand{\mH}{\mathcal{H}}
\newcommand{\mZ}{\mathbb{Z}}

\title{Generalized $S^1$-stability theorem}

\author{Tongrui Wang}
\address{Institute for Theoretical Sciences, Westlake Institute for Advanced Study, Westlake University, Hangzhou, Zhejiang, 310024, China}
\email{wangtongrui@westlake.edu.cn}

\author{Xuan Yao}
\address{Cornell University, Department of Mathematics, Ithaca, New York 14850}
\email{xy346@cornell.edu}

\date{September 2023}

\begin{document}

\maketitle

\begin{abstract}
    We use the equivariant $\mu$-bubbles technique to prove that for any compact manifold $M^n$
    with non-empty boundary, $n\in\{3,5,6\}$, the Yamabe invariant of $M^n$ is positive if
    and only if the Yamabe invariant of $M^n\times S^1$ is positive. 
    This generalized the $S^1$-stability conjecture of Rosenberg to compact manifolds with boundary. 
\end{abstract}

\section{Introduction}
The study of scalar curvature holds an influential place in the field of differential geometry. 
An intriguing branch in the study of scalar curvature is the topology of manifolds with metrics of positive scalar curvature, which can be fully characterized by the positivity of Yamabe invariant. 

A well-known result in the study of scalar curvature was established by R.Schoen and S.-T. Yau as well as by M.Gromov and B.Lawson:
\begin{theorem}[Geroch Conjecture, \cite{schoen28structure}\cite{gromov1980spin}]\label{conj:geroch}
    The $n$-torus does not admit a Riemannian metric of positive scalar curvature.
\end{theorem}

In \cite{schoen1987structure}, Schoen-Yau extended the Geroch conjecture to the so-called {\em $K(\pi,1)$ conjecture}, and they also outlined the proof in dimension $4$.
\begin{conjecture}[$K(\pi,1)$ Conjecture]\label{conj: aspherical}
   Any $n$-dimemsional $K(\pi,1)$ closed manifold does not admit a metric with positive scalar curvature.
\end{conjecture}

This conjecture was confirmed in dimension $n\in\{4,5\}$ independently by O.Chodosh and C.Li \cite{chodosh2020generalized}
and by M.Gromov \cite{gromov2020no}. A mapping version of this conjecture was later proven by O.Chodosh, C.Li, and L.Yevgeny in \cite{chodosh2023classifying} for dimension $n\in\{4,5\}$.

It is natural to ask whether one can generalize $K(\pi,1)$ Conjecture to manifolds with boundary, which leads us to the following conjecture: 
\begin{conjecture}\label{conj:weak aspherical}
    For any compact $n$-manifold $M^n$ with non-empty boundary, if the double of $M^n$ is a $K(\pi,1)$
    closed manifold, then $\sigma(M^n)$ is not positive.
\end{conjecture}

\begin{remark} $ $
    \begin{enumerate} 
        \item Conjecture \ref{conj:weak aspherical} is a much weaker version of Conjecture \ref{conj: aspherical}, 
              it is true provided Conjecture \ref{conj: aspherical} is true.
        \item  For compact manifolds, one can easily find $K(\pi,1)$ manifolds with positive Yamabe invariants, for example, the $n$-disks. 
    \end{enumerate}
\end{remark}

Note that an $S^1$-bundle of $K(\pi,1)$ manifold is still a $K(\pi,1)$ manifold and vice versa. 
It is reasonable to investigate the relation between the positivity of Yamabe invariant of a manifold and the positivity of Yamabe invariant of $S^1$-bundle over that manifold. 

In \cite{rosenberg2006manifolds}, J.Rosenberg proposed the following {\em $S^1$-stability Conjecture} for closed manifolds. 
\begin{conjecture}[$S^1$-stability]\label{conj:s1 stability}
Suppose $M^n$ is a closed, connected $n$-manifold, then the Yamabe invariants $\sigma$ of $M$ and $M\times S^1$ satisfy
\begin{align*}
    \sigma(M^n)>0\iff \sigma(M^n\times S^1)>0.
\end{align*}
\end{conjecture}
\begin{remark} $ $
    \begin{enumerate}
        \item This conjecture fails when $n=4$ due to a counterexample depending on the Seiberg-Witten equations (cf. \cite[Remark 1.25]{rosenberg2006manifolds});
        \item This conjecture is true for $n=3$ by the classification of $3$-manifolds with positive scalar curvature (see for instance the lecture notes of O. Chodosh \cite{chodosh2021STABLEMS}).
        \item R{\"a}de \cite{rosenberg2006manifolds} proved this conjecture for dimension $n \in\{5,6\}$. We claim that R{\"a}de's proof can be extended to $n=7$, since in Gromov's notes \cite{gromov2019four}, he claimed that
              one can obtain the generic regularity of $\mu$-bubble in dimension $8$ by following N.Smale's work on the generic regularity of minimal hypersurfaces in dimension $8$. 
    \end{enumerate}
\end{remark}

It is natural to inquire whether one can generalize Conjecture \ref{conj:s1 stability} to compact $n$-manifolds with boundary.
In this paper, we give an affirmative answer to this question in dimension $n\in\{3,5,6\}$.

\begin{theorem}[$S^1$-stability on compact $n$-manifolds]\label{thm:compact s1 stability}
    Suppose $M^n$ is a compact manifold with non-empty boundary, $n\in\{3,5,6\}$, then the Yamabe invariants $\sigma$ of $M$ and $M\times S^1$ satisfy
    \begin{align}
        \sigma(M^n)>0\iff \sigma(M^n\times S^1)>0.
    \end{align}
\end{theorem}

M.Gromov and B.Lawson \cite{gromov1980spin} first observed that if a compact manifold $M^n$ admits a metric with positive scalar curvature
and mean-convex boundary, then $DM^n$ the double of $M^n$ admits a $Z_2$-metric with positive scalar curvature. 

Inspired by Gromov-Lawson's work, we consider the double of the compact manifold and prove an equivariant version
of $S^1$-stability conjecture in dimension $n\in\{5,6\}$ to prove Theorem \ref{thm:compact s1 stability}.
\begin{theorem}[Equivariant $S^1$-stability]\label{thm:equivariant s1 stability}
    Let $M^n$ be a closed $n$-manifold, where $n\in\{5,6\}$. Suppose the finite group $\mathbb Z_2$ acts on $M^n$ by reflection.
    Then the $\mZ_2$-equivariant Yamabe invariants $\sigma_{\mZ_2}$ of $M$ and $M\times S^1$ satisfy
    \begin{align}
        \sigma_{\mathbb Z_2}(M^n)>0\iff \sigma_{\mathbb Z_2}(M^n\times S^1)>0,
    \end{align}
    where $\mathbb Z_2$ only acts on the $M^n$ of $M^n\times S^1$, the action is the same as that on $M^n$.
\end{theorem}
\begin{remark}
There are some obstacles in directly generalizing R{\"a}de's proof for the closed
$S^1$-stability Theorem \cite{rosenberg2006manifolds} to the compact manifold with non-empty boundary.
Our proof suggests that the study of Yamabe invariant of compact manifolds can be converted to the study of the equivariant 
Yamabe invariant of closed manifolds.
 \end{remark}

\subsection{Outline of the proofs}
 We first give the detailed proof of equivariant $\mu$-bubble separation Theorem which was stated 
 in Gromov's notes \cite{gromov2019four}. As a Corollary, we prove the equivariant $\frac{2\pi}{n}$-inequality.
 Then we prove an equivariant surgery Proposition, combined with the $G$-invariant version of Gromov-Lawson and Schoen-Yau's surgery result, together with the equivariant $\frac{2\pi}{n}$-inequality, we finish the proof of Theorem \ref{thm:equivariant s1 stability} and Theorem \ref{thm:compact s1 stability}.

\subsection{Organization}
In Section 2, we introduce necessary preliminaries. In Section 3, we begin by establishing the existence and regularity of equivariant $\mu$-bubbles in dimensions less than $8$. Then we proceed to prove the equivariant
$\mu$-bubble separation Theorem and as a corollary, we obtain the equivariant $\frac{2\pi}{n}$-inequality.
In Section 4, we first show an equivariant surgery Proposition and then complete the proof of the main theorems.
\subsection*{Acknowledgements}
Both of the authors would like to thank Professor Xin Zhou for his helpful discussions on this topic and constant encouragement. They would like to thank Professor Yuguang Shi for his suggestion on considering the double of the compact manifold and his interest in this topic. 
The authors would also like to thank Shihang He for his helpful discussion. 
T.W. is partially supported by China Postdoctoral Science Foundation 2022M722844. X.Y is supported by NSF grant DMS-1945178.

\section{Preliminaries}

In this section, we provide essential preliminary concepts concerning Yamabe invariants, doubling manifolds, and Riemannian bands. 

\subsection{Yamabe Invariant}~

As a higher-dimensional generalization for Euler characteristic, the {\em Yamabe invariant} (or {\em $\sigma$-invariant}) for a closed manifold is a min-max defined value associated with a smooth manifold that is preserved under diffeomorphisms.
This differential-topological invariant was first introduced independently by O.Kobayshi \cite{kobayashi1986scalar} and R.Schoen \cite{schoen2006variational}. 
The definition of Yamabe invariant of compact manifolds was first written down by J.Escobar \cite{escobar1996conformal}. 
For manifolds acted diffeomorphically on by a compact Lie group, Hebey and Vaugon \cite{hebey1993probleme} proposed the definition of {\em equivariant Yamabe invariant} in their study of equivariant Yamabe problem. 

For the convenience of readers, we collect some related definitions in this subsection.
\begin{definition}
    Suppose $M^n$ is a closed $n$-manifold, the {\em Yamabe invariant} of $M^n$ is defined as
    \begin{align*}
        \sigma(M^n):=\sup_{[g]\in \mathcal C}Y(M,[g]),
    \end{align*}
    where $\mathcal C$ is the set of all conformal classes of metrics on $M^n$, and $Y(M,[g])$ is the {\em Yamabe constant} for the conformal class $[g]$, defined as
    \begin{align*}
        Y(M,[g]):=\inf_{\tilde{g}\in [g], \text{Vol}(\tilde{g})=1}\int_{M^n}R_{\tilde{g}}dV_{\tilde{g}},
    \end{align*}
    where $R_{\tilde{g}}$ is the scalar curvature of $(M^n,\tilde{g})$.
\end{definition}

One can define the $G$-invariant Yamabe invariant by restricting the conformal classes to the $G$-invariant
conformal classes.

\begin{definition}
    Suppose $M^n$ is a closed $n$-manifold acted diffeomorphically upon by a compact Lie group $G$. The {\em $G$-invariant Yamabe invariant} of $M^n$ is defined as
    \begin{align*}
        \sigma_{G}(M^n):=\sup_{[g]_{G}\in \mathcal C_{G}}Y_{G}(M,[g]_{G}),
    \end{align*}
    where $\mathcal C_{G}$ is the set of all $G$-invariant conformal classes of metrics on $M^n$, and $Y_{G}(M,[g]_G)$ is the {\em $G$-invariant Yamabe constant} for the $G$-invariant conformal class $[g]_{G}$, defined as
    \begin{align*}
        Y_{G}(M,[g]_G):=\inf_{\tilde{g}\in [g]_{G}, \text{Vol}(\tilde{g})=1}\int_{M^n}R_{\tilde{g}}dV_{\tilde{g}},
    \end{align*}
    where $R_{\tilde{g}}$ is the scalar curvature of $(M^n,\tilde{g})$.
\end{definition}

Then we introduce the definition of Yamabe invariant of a compact manifold with non-empty boundary.
\begin{definition}
    Suppose $M^n$ is a compact manifold and its boundary $\partial M$ is non-empty, then the {\em Yamabe invariant} of $M^n$ is defined as
    \begin{align*}
        \sigma(M^n):=\sup_{[g]\in \mathcal C}Y(M,\partial M,[g]),
    \end{align*}
    where $\mathcal C$ is the set of all conformal classes of metrics on $M^n$, and $Y(M,\partial M,[g])$ is the {\em Yamabe constant} for the conformal class $[g]$, defined as
    \begin{align*}
        Y(M,\partial M,[g]):=\inf_{\tilde{g}\in [g], \text{Vol}(\tilde{g})=1}\left[\int_{M^n}R_{\tilde{g}}dV_{\tilde{g}}+2(n-1)\int_{\partial M}H_{\tilde{g}}dA_{\tilde{g}}\right].
    \end{align*}
\end{definition}

\subsection{Topological Preliminaries} $ $

We introduce some basic topological concepts as follows.
\begin{definition}[Double of a manifold]
    Given a compact manifold $M^n$, we use $DM^n$ to denote the double of $M^n$, and it is defined as the smoothing of the boundary of $M^n\times[0,1]$.
\end{definition}

\begin{remark}
    Conceptually, one can view the double of a compact manifold as gluing along the boundary. 
     Since we mainly consider the manifolds with positive scalar curvature, we often refer to \cite[Theorem 5.7]{gromov1980spin} for the construction of doubled manifolds.
\end{remark}

By taking the double of a compact manifold $N$, we obtain a closed manifold $DN$ with $\mZ_2$-symmetry. 
In general, we also say a (Riemannian) manifold $M$ is a closed {\em $G$-manifold} if $G$ is a compact Lie group acting by diffeomorphisms (isometries) on $M$. 

In such a $G$-manifold, Wassermann \cite{wasserman1969equivariant} first formulated the Morse theory in an equivariant setting and constructed $G$-handles decompositions. 

Specifically, given a closed Riemannian $G$-manifold $M$ and a smooth $G$-invariant function $f:M\to \R$ (i.e. $f(g\cdot p)=f(p)$ for all $g\in G, p\in M$), an orbit $G\cdot p$ is said to be a {\em critical orbit} of $f$ if for one (and hence any) point $q\in G\cdot p$ the differential $D_pf$ is zero. 
In addition, a critical orbit $G\cdot p$ is said to be {\em non-degenerated} if for each $q\in G\cdot p$, the Hessian $D^2_qf$ is non-degenerated on the normal space $(T_qG\cdot p)^\perp$ of $G\cdot p$ at $q$. 
The {\em index} (resp. {\em coindex}) of $f$ at the critical orbit $G\cdot p$ is then defined by the index (resp. coindex) of the restricted Hessian $D^2_qf\llcorner (T_qG\cdot p)^\perp$. 
Moreover, $f$ is said to be a {\em $G$-Morse function} if it only has non-degenerate critical orbits. 

\begin{remark}\label{Rem: density Morse function}
    Note this definition is not vacuous since the set of $G$-Morse functions is dense (and open) in the set of smooth $G$-invariant functions equipped with the $C^\infty$-topology (cf. \cite{wasserman1969equivariant}). 
\end{remark}

Since we aim to construct equivariant metrics of positive scalar curvature, we also need the following {\em special $G$-Morse functions} (introduced by Mayer in \cite{mayer1989g}) to control the codimensions of $G$-handles using the arguments in \cite{hanke2008positive}. 

\begin{definition}\label{defn:special morse}
    Let $M$ be a closed $G$-manifold. A $G$-Morse function
    \begin{align*}
        f:M\to\mathbb R
    \end{align*}
    is called {\em special} if for each critical orbit $\mathcal O$, the index of $f$
    at $\mathcal O$ is equal to the index of the restricted $G$-Morse function
    \begin{align*}
        f|_{M_{(H)}}\to\mathbb R
    \end{align*}
    at $\mathcal O$, where $(H)$ is the {\em isotropy type} of $\mathcal O$ defined by the conjugacy class of the isotropy group $G_p:=\{g\in G: g\cdot p=p\}$ for any $p\in\mathcal{O}$, and $M_{(H)}:=\{x\in M: G_x\in (H)\}$ is a $G$-submanifold of $M$ known as the {\em $(H)$-orbit type stratum}.
\end{definition}

At the end of this section, we introduce the definitions associated with Riemannian bands.
\begin{definition}[Band and $G$-invariant band]
    A {\em band} is a compact manifold $X$, and a boundary decomposition
    \begin{align*}
        \partial X=\partial_-X\sqcup \partial_+X.
    \end{align*}
    If $X$ is equipped by a Riemannian metric, we call it a {\em Riemannian band}, and we define the width 
    of $X$ as:
    \begin{align*}
        width(X):=\text{dist}(\partial_{-}X,\partial_+X).
    \end{align*}
    If $X$ is acted diffeomorphically (resp. isometrically) upon by a compact Lie group $G$, and $G\cdot \partial_{\pm}X=\partial_{\pm}X$,
    then we call $X$ a {\em $G$-invariant band} (resp. {\em $G$-invariant Riemannian band}).
\end{definition}

A {\em separating hypersurface} of a band $X$ is a closed hypersurface $\Sigma$ in $X$ such that $\partial_-X$
and $\partial_+X$ are contained in the different connected components in $X\setminus \Sigma$.

In addition, we define the Morse function (resp. $G$-Morse function) on a (resp. $G$-invariant) band $X$ by smooth (resp. $G$-invariant) functions $f:X\to \R$ satisfying
\begin{itemize}
    \item $f$ only has non-degenerate critical points (resp. orbits); 
    \item $f(X)\subset [0, 1]$, $f\llcorner \bd_-X = 0$, and $f\llcorner \bd_+X=1$;
    \item the critical values of $f$ are different from $0$ and $1$;
\end{itemize}
In parallel to Definition \ref{defn:special morse}, one can similarly define special $G$-Morse functions in a $G$-invariant band. 
Moreover, the above definitions can also be directly extended to a ($G$-invariant) cobordism so that the analog of the density result (Remark \ref{Rem: density Morse function}) is still valid.

\section{Equivariant $\frac{2\pi}{n}$-inequality}

In this section, we give a detailed proof of {\em equivariant $\frac{2\pi}{n}$-inequality} which was essentially stated and outlined in Gromov's notes \cite{gromov2019four}.

In \cite{gromov2019four}, 
Gromov proved the so-called {\em $\mu$-bubble separation Theorem}, and as a Corollary, he proved the $\frac{2\pi}{n}$-inequality from the 
$\mu$-bubble separation Theorem. Gromov also stated the equivariant version of the $\mu$-bubble separation
Theorem in \cite{gromov2019four}, we state his result and rewrite his proof in detail.
\begin{theorem}[Equivariant $\mu$-bubble separation, Gromov \cite{gromov2019four}]\label{thm:equivariant mu bubble separation} $ $
    
    Let $(X,g_X)$ be an $n$-dimensional Riemannian band acted isometrically upon by a compact Lie group $G$
    with $G\cdot \bd_\pm X=\bd_\pm X$ and $3\leq {\rm codim}(G\cdot x)\leq 7$ for all $x\in X$.
    Suppose the scalar curvature of $X$ satisfies
    \begin{align*}
        R_{g_X}(x)\geq w(x)+w_1,
    \end{align*}
    where $w(x)\geq 0$ is a continuous equivariant function on $X$, and $w_1$ is a positive constant related to the width $d$ of $X$ by
    \begin{align*}
        w_1d^2>\frac{4(n-1)\pi^2}{n}.
    \end{align*}

    Then there exists a smooth $G$-invariant separating hypersurface $Y\subset X$ and a $G$-invariant positive function $\phi$ on $Y$, 
    such that the scalar curvature of the metric $g_{\phi}=g_Y+\phi(y)^2dt^2$
    on $Y\times \mathbb{R}$ is bounded from below by 
    \begin{align*}
        R_{g_{\phi}}((y,t))\geq w(y).
    \end{align*}
\end{theorem}

We obtain the equivariant $\frac{2\pi}{n}$-inequality using the Equivariant $\mu$-bubble separation Theorem \ref{thm:equivariant mu bubble separation} as
follows.
\begin{proposition}[Equivariant $\frac{2\pi}{n}$-inequality]\label{prop:equivariant 2pi/n}
    Suppose $(X^n,g)$ is a  compact Riemannian band acted upon by a compact
    Lie group $G$  with $G\cdot \bd_\pm X=\bd_\pm X$ and $3\leq {\rm codim}(G\cdot x)\leq 7$ for all $x\in X$.
    Assume further that there is no $G$-invariant separating 
    hypersurface admitting a $G$-invariant positive scalar curvature and 
    \begin{align*}
        R_g(x)\geq n(n-1),
    \end{align*}
    then 
    \begin{align*}
        \text{width}(X)\leq \frac{2\pi}{n}.
    \end{align*}
\end{proposition}
\begin{proof}
    We prove this Proposition by argument of contradiction.
    Assume the opposite, which means there is a Riemannian band $(X^n,g)$
    satisfying the assumptions but it has the width greater than $\frac{2\pi}{n}$.
    Then Theorem \ref{thm:equivariant mu bubble separation} implies that there is 
    a $G$-invariant separating hypersurface $Y$ and a $G$-invariant positive function $\phi$ defined on $Y$, such that 
    $(Y\times\mathbb{R},g_{\phi})$ has scalar curvature at least $\epsilon$ for some $\epsilon>0$. 
    
    Recall the scalar curvature of a warped product $g_{\phi}=g_Y+\phi(y)^2dt^2$ is:
    \begin{align}\label{warped product}
        R_{g_{\phi}}(y,t)=R_{g_Y}(y)-2\frac{\Delta_{g_Y}\phi}{\phi},
    \end{align}
    which implies that $(Y^{n-1},\phi^{\frac{4}{n-3}}g_{Y})$ is a $G$-invariant Riemannian manifold
    with positive scalar curvature.

    Therefore, $Y$ admits a $G$-invariant metric with positive scalar curvature, and contradiction is achieved. 
\end{proof}

\subsection{Preliminaries on the Equivariant $\mu$-bubble}
Before starting the proof of Theorem \ref{thm:equivariant mu bubble separation}, we first introduce some results on the equivariant $\mu$-bubble.

Consider an $n$-dimensional oriented Riemannian band $(X,g_{_X})$ acted isometrically on by a compact Lie group $G$ so that both $\bd_+X$ and $\bd_-X$ are $G$-invariant. 
Denote by 
\begin{itemize}
    \item $\C(X)$: the set of all Caccioppoli sets in $X$ which contain an open neighborhood of $\bd_-X$ and are disjoint from $\bd^+X$;
    \item $\C^G(X)$: the set of all $\Omega\in\C(X)$ with $g\cdot \Omega=\Omega$ for all $g\in G$.
\end{itemize}
Given any smooth function $h$ on $X$, let 
\[\A_h(\Omega) := \mH^{n-1}(\partial\Omega \cap \interior(X) ) - \int_{\Omega} h d\mH^n,\qquad \forall\Omega\in\C(X), \]
be the $\A_h$-functional on $\C(X)$, where $\bd\Omega$ denotes the reduced boundary of $\Omega$. 
\begin{definition}\label{Def: mu-bubble}
    A Caccioppoli set $\Omega\in\C(X)$ is said to be a {\em $\mu$-bubble}, if 
    \begin{align*}
        \A_h(\Omega) = \inf_{\hat{\Omega}\in\C(X)} \A_h(\hat{\Omega}).
    \end{align*}
\end{definition}

Using the averaging trick in \cite[Lemma 6.2]{wang2022existence} (initiated in \cite{lawson1972equivariant}), we show the following existence result on the equivariant $\mu$-bubble.
\begin{proposition}\label{Prop: equivariant mu-bubble}
    Suppose $3\leq {\rm codim}(G\cdot x)\leq 7$ for all $x\in X$, $h$ is $G$-invariant, and $H(\bd_{\pm}X)>\pm h$ on $\bd_\pm X$, where $H(\bd_{\pm}X)$ is the mean curvature of $\bd_{\pm}X$ with respect to the inward unit normal of $X$.
    Then there exists a smooth $G$-invariant $\mu$-bubble $\Omega\in\C(X)$ with $\A_h(\Omega) = \inf_{\hat{\Omega}\in\C_G(X)} \A_h(\hat{\Omega}) = \inf_{\hat{\Omega}\in\C(X)} \A_h(\hat{\Omega})$.
\end{proposition}
\begin{proof}
    Since $\A_h(\hat{\Omega}) \geq \sup|h| \mH^n(X) > -\infty$, we have 
    \[ \I_G:=\inf_{\hat{\Omega}\in\C_G(X)} \A_h(\hat{\Omega}) \quad \geq\quad  \I:= \inf_{\hat{\Omega}\in\C(X)} \A_h(\hat{\Omega}) \quad>-\infty.\]
    Let $\Omega_{\pm}^t$ be the $t$-neighborhood of $\bd_{\pm}X$ for $t>0$ small enough. 
    It follows from the $G$-invariance of $\bd_\pm X$ that $\Omega_{\pm}^t$ are $G$-invariant too. 
    By the arguments in \cite[Lemma 4.2]{rade}, we have $\A_h((\hat{\Omega}\cup\Omega_-^t)\setminus\Omega_+^t) < \A_h(\hat{\Omega})$ for all $\hat{\Omega}\in\C(X)$ and $t>0$ small. 
    Hence, for any $G$-invariant minimizing sequence $\{\hat{\Omega}_k\}\subset\C_G(X)$ with $\A_h(\hat{\Omega}_k)\to\I_G$, we can replace $\hat{\Omega}_k$ by $(\hat{\Omega}_k\cup\Omega_-^t)\setminus\Omega_+^t$
     and apply the compactness theorem to obtain a limit $G$-invariant Caccioppoli set $\Omega = \lim_{k\to\infty} (\hat{\Omega}_k\cup\Omega_-^t)\setminus\Omega_+^t$ with $\A_h(\Omega)\leq \I_G$. 
    By our constructions, $\Omega$ contains an open neighborhood of $\bd_-X$ and is disjoint from $\bd_+X$, which further implies $\Omega\in\C_G(X)$ and $\A_h(\Omega)=\I_G$. 

    \begin{claim}\label{Claim: equivariant mu-bubble}
        $\I_G=\I$, i.e. $\A_h(\Omega)\leq \A_h(\hat{\Omega})$ for all $\hat{\Omega}\in\C(X)$. 
    \end{claim}
    \begin{proof}[Proof of Claim \ref{Claim: equivariant mu-bubble}]
        Given any $\hat{\Omega}\in\C(X)$, define 
        \[f(x):=\int_G 1_{\Clos(\hat{\Omega})}(g\cdot x)d\mu(g),\]
        where $1_{\Clos(\hat{\Omega})}$ is the characteristic function of $\Clos(\hat{\Omega})$, and $\mu$ is the bi-invariant Haar measure on $G$ normalized to $\mu(G)=1$. 
        Since $\hat{\Omega}\in \C(X)$, we have 
        \begin{equation}\label{Eq: averaged charact func}
            \mbox{$f=1$ near $\bd_-X$}\qquad {\rm and} \qquad \mbox{$f=0$ near $\bd_+X$}.
        \end{equation}
        Additionally, note $f$ is a $G$-invariant function on $X$, which is also upper-semicontinuous by Fatou’s Lemma. 
        Hence, 
        \[\hat{\Omega}_\lambda :=f^{-1}[\lambda,1] \]
        are $G$-invariant closed sets in $X$.   

        Consider the current $E_f:= f\cdot [[X]]$, where $[[X]]$ is the integral current induced by $X$. 
        Then for any $n$-form $\omega$ on $X$, we have
        \begin{align*}
            \bd E_f(\omega) &= \int_X\int_G \langle d\omega, \xi\rangle 1_{\Clos(\hat{\Omega})}(g\cdot x)d\mu(g) d\mH^{n}(x)\\
            &= \int_G\int_X \langle d\omega, \xi\rangle 1_{g^{-1}\cdot \Clos(\hat{\Omega})}(x) d\mH^{n}(x)d\mu(g)\\
            &= \int_G \bd(g^{-1}\cdot\hat{\Omega})(\omega) d\mu(g),
        \end{align*}
        which implies 
        $ \|E_f\|(\interior(X)) \leq \int_G \|\bd(g^{-1}\cdot\hat{\Omega})\|(\interior(X)) d\mu(g) = \|\bd \hat{\Omega}\|(\interior(X)) $, and $E_f$ is a normal current. 
        It then follows from \cite[4.5.9 (12)]{federer2014geometric} that $\bd\hat{\Omega}_\lambda$ is rectifiable for almost all $\lambda\in [0,1]$. 
        Combining with (\ref{Eq: averaged charact func}), we have $\hat{\Omega}_\lambda\in\C_G(X)$ for almost all $\lambda\in [0,1]$. 
        Hence, 
        \begin{equation}\label{Eq: average Ah 1}
            \A_h(\Omega) = \I_G \leq \int_0^1 \A_h(\hat{\Omega}_\lambda) d\lambda. 
        \end{equation}
        
        Next, by \cite[4.5.9 (13)]{federer2014geometric} and the above computations, we have $\bd E_f = \int_0^1\bd \hat{\Omega}_\lambda d\lambda$ and 
        \begin{equation}\label{Eq: average Ah 2.1}
            \int_0^1 \|\bd \hat{\Omega}_\lambda\|(\interior(X)) d\lambda = \|\bd E_f\|(\interior(X)) \leq  \|\bd \hat{\Omega}\|(\interior(X)). 
        \end{equation}
        In addition, we see from the definition of integration that 
        \begin{align*}
            \int_X fh d\mH^n(x) &= \int_{\{h\geq 0\}} fh d\mH^n(x) - \int_{\{h\leq 0\}} f(-h)\mH^n(x) \\ 
            &=  \int_0^1 (\mH^n\llcorner h_+) ( f^{-1}[\lambda, 1] ) d\lambda  - \int_0^1 (\mH^n\llcorner h_-) ( f^{-1}[\lambda, 1] ) d\lambda \\
            &= \int_0^1 \int_{f^{-1}[\lambda, 1]} h_+ - h_- d\mH^n(x)d\lambda= \int_0^1 \int_{\hat{\Omega}_\lambda} h d\mH^n(x)d\lambda,
        \end{align*}
        where $h_+ = \max\{h,0\}$ and $h_-=-\min\{h,0\}$. 
        On the other hand, we also have 
        \begin{align*}
            \int_X fh d\mH^n(x) &= \int_X\int_G h(g\cdot x)\cdot 1_{\Clos(\hat{\Omega})}(g\cdot x)d\mu(g) d\mH^n(x) \\
            &= \int_G\int_X h(x)\cdot 1_{g^{-1}\cdot \Clos(\hat{\Omega})}(x) d\mH^n(x) d\mu(g) \\
            &= \int_G (\mH^n\llcorner h)(g^{-1}\cdot \hat{\Omega}) d\mu(g) = \int_G \int_{\hat{\Omega}} h(x) d\mH^n(x) d\mu(g) \\
            &= \int_{\hat{\Omega}} h d\mH^n.
        \end{align*}
        Together, we have $\int_0^1 \int_{\hat{\Omega}_\lambda} h d\mH^n(x)d\lambda = \int_X fh d\mH^n(x) = \int_{\hat{\Omega}} h d\mH^n $, which implies the following inequality by (\ref{Eq: average Ah 2.1}):
        \begin{align}\label{Eq: average Ah 2}
            \int_0^1 \A_h(\hat{\Omega}_\lambda) d\lambda &= \int_0^1  \|\bd \hat{\Omega}_\lambda\|(\interior(X)) d\lambda  -  \int_0^1\int_{\hat{\Omega}_\lambda} h d\mH^n(x)  d\lambda \nonumber \\
            & \leq \|\bd \hat{\Omega}\|(\interior(X)) - \int_{\hat{\Omega}} h d\mH^n = \A_h(\hat{\Omega}).
        \end{align}
        Combining (\ref{Eq: average Ah 1}) with (\ref{Eq: average Ah 2}), we conclude $\A_h(\Omega)\leq\A_h(\hat{\Omega})$, and thus $\I_G=\I$. 
    \end{proof}
    The above claim implies $\Omega$ is a $G$-invariant $\mu$-bubble. Noting the regular/singular set of $\bd\Omega$ is $G$-invariant, the regularity
     of $\bd\Omega$ then follows immediately from the regularity theorem \cite[Theorem 2.2]{zhou2020existence} and the dimension assumption: $3\leq {\rm codim}(G\cdot x)\leq 7$ for all $x\in X$. 
\end{proof}

\subsection{Equivariant $\mu$-bubble separation Theorem}

Now, we write the proof of Theorem \ref{thm:equivariant mu bubble separation} following Gromov's notes \cite{gromov2019four}.
\begin{proof}[Proof of Theorem \ref{thm:equivariant mu bubble separation}]
    Without loss of generality, we can assume $w_1=n(n-1)$. By the 
    assumption, we know the width $d>\frac{2\pi}{n}$. Using the distance function, one can construct
    a $G$-invariant function $f: X\to [-\frac{\pi}{n},\frac{\pi}{n}]$, satisfying:
        \begin{align*}
            Lip(f)<1, \quad f(\partial_{-}X)=-\frac{\pi}{n}, \quad f(\partial_{+}X)=\frac{\pi}{n}.
        \end{align*}
    One can find a detailed construction of this function in \cite[Lemma 7.2]{cecchini2022scalar}, which can be easily modified to the $G$-invariant case.

    Suppose $(T^n,g_0)$ is the standard flat metric on $n$-torus. Let $M=T^n\times (-\frac{\pi}{n},\frac{\pi}{n})$ equipped 
    with the product metric 
    \begin{align*}
        g_M=\varphi(t)^2g_0+dt^2,
    \end{align*}
    where 
    \begin{align*}
        \varphi(t)=e^{\int_{-\frac{\pi}{n}}^t-\tan (\frac{n}{2}s)ds}.
    \end{align*}
    We can obtain the scalar curvature $R_{g_M}=n(n-1)$ and the mean curvature $H(T^n\times\{t\})=-(n-1)\tan(\frac{n}{2}t)$
    (with respect to the unit normal pointing towards $T^n\times\{-\frac{\pi}{n}\}$) by 
    straightforward computations (cf. \cite[\S 2.4]{gromov2019four}). 

    Now, we define $h: \interior(X) \to\mathbb R$ as 
    \begin{align*}
        h(x)=-(n-1)\tan(\frac{n}{2}f(x)). 
    \end{align*}
    Note $h\to\pm\infty$ as $x\to \bd_{\mp}X$, while the mean curvature of $\bd X$ is uniformly bounded. 
    Therefore, after shrinking $X$ to its interior $X\setminus B_\epsilon(\bd X)$ with $\epsilon>0$ 
    sufficiently small, we can apply Proposition \ref{Prop: equivariant mu-bubble} in $X\setminus B_\epsilon(\bd X)$ to obtain a smooth $G$-invariant $\mu$-bubble $Y\subset \interior (X\setminus B_\epsilon(\bd X))$, which is also a separating hypersurface in $X$. 

    Since $Y$ is the minimizer of the $\A_h$-functional, the second variation
    of $\A_h$ at $Y$ is non-negative. 
    More precisely, we obtain
    \begin{align*}
        \int_Y|\nabla^Y\psi|^2+\frac{1}{2}R_{g_Y}\psi^2\geq \int_{Y}\frac{1}{2}\left(R_g+\frac{n}{n-1}h^2+2g(\nabla^X h,\nu)\right)\psi^2,
    \end{align*}
    where $\psi$ is any smooth function on $Y$.

    Note that, for any $y\in Y$,
    \begin{align*}
        R_g+\frac{n}{n-1}h^2+&2g(\nabla^X h,\nu)\\
        &> w(y)+n(n-1)\left(1+\tan^2(\frac{n}{2}f(y))\right)-n(n-1)\sec^2(\frac{n}{2}f(y))\\
         &=w(y).
    \end{align*}
    Therefore, we conclude that $-\Delta_Y+\frac{1}{2}R_Y-\frac{1}{2}w(y)$ is a positive operator, hence the first eigenfunction 
    of this operator, denoted as $\phi$, is positive and $G$-invariant (by the strong maximum principle). 
    With \eqref{warped product}, we know $R_{g_{\phi}}\geq w(y)$ for any $(y,t)\in Y\times\mathbb R$. 
\end{proof}

\section{Proof of the Main Theorem}
In this section, we complete the proof of Theorem \ref{thm:equivariant s1 stability}. As a Corollary,
we obtain Theorem \ref{thm:compact s1 stability}. 

Before starting the proof of Theorem \ref{thm:equivariant s1 stability}, we prove an equivariant version of the surgery theory 
Proposition in R{\"a}de \cite{rade}, which is an essential ingredient in the proof of the Main Theorem. The non-equivariant version of this proposition is stated and proven in R{\"a}de's proof on the closed $S^1$-stability paper \cite{rade}. However, his proof 
cannot be directly generalized to the equivariant case, since there is no equivariant version of the handle cancellation theorem. 
As a matter of fact, the equivariant version of the handle cancellation theorem fails in general.

\begin{proposition}\label{prop:equivariant surgery}
    Suppose $Y^n$ is a closed oriented manifold of dimension at least $5$, 
    $X=Y\times [-1,1]$. Let the finite group $\mathbb Z_2$ act on $Y$ by reflection with $n$-dimensional $\mZ_2$-fixed points set ${\rm Fix}(Y)\times [-1,1]$ on $X$. 
    Assume further that there is a $\mathbb Z_2$-invariant separating hypersurface $\Sigma$ in $X$,
    then $Y$ can be obtained by $\Sigma$ by a finite sequence of $\mZ_2$-equivariant surgeries in codimension at least $3$.
    
\end{proposition}

To prove the above Proposition, we first follow R{\"a}de's proof of Proposition 6.4 in \cite{rade}, then we use an equivariant version of Morse function construction which is essentially introduced in the appendix of X.Zhou's work \cite{10.4310/jdg/1427202766} with minor modifications, finally, we adopt part of Hanke's proof of Theorem 15 \cite{hanke2008positive} to complete the above equivariant version of surgery theory Proposition.

\begin{proof}
   We use $W$ to denote the connected component of $X\setminus\Sigma$ which contains $\partial_-X\cong Y$. 
   Then $W$ is a $\mZ_2$-invariant cobordism of $Y$ and $\Sigma$ with $\mZ_2$-invariant retract map $r: W\to Y$, where $r$ is the restriction of the projection map from $X$ to $Y$ on $W$.
   If we further restrict to the fixed point set, we obtain a cobordism $\text{Fix}(W)$ of $\text{Fix}(Y)$ and $\text{Fix}(\Sigma)$ with a retract map 
   $\tilde{r}:\text{Fix}(W)\to\text{Fix}(Y)$, where $\tilde{r}=r|_{\text{Fix}(W)}$.

   Applying R{\"a}de's argument \cite[Proposition 6.4]{rade} to $\text{Fix}(W)$, we can kill the kernel of 
   \begin{align*}
    \pi_1(\tilde{r}):\pi_1(\text{Fix}(W))\to \pi_1(\text{Fix}(Y))
   \end{align*}
   by surgeries in the interior of $W$ (not ${\rm Fix}(W)$). 
   Note the above surgery around a closed curve $\gamma\subset \interior(\text{Fix}(W))$ is naturally $\mZ_2$-equivariant by the $\mZ_2$-invariance of $\gamma$.
   Hence, we obtain a $\mathbb Z_2$-cobordism $W_1$ between $Y$ and $\Sigma$, and a $\mathbb Z_2$-retract map $r_1:W_1\to Y$ which is $2$-connected restricted to $\text{Fix}(W_1)$.

    Denote by $W_1^{\text{prin}}:=W_1\setminus {\rm Fix}(W_1)$ and $Y^{\text{prin}}:=Y\setminus {\rm Fix}(Y)$ the unions of principal orbits (each of which contains $2$-components with $\mZ_2$ acting by permutations).
    Applying R{\"a}de's argument \cite[Proposition 6.4]{rade} to one component $W_1^{\text{prin}}/\mZ_2$ of $W_1^{\text{prin}}$ and reflecting to the other, 
    we can kill the kernel of $\pi_i(r_1):\pi_i(W_1^{\text{prin}})\to \pi_i(Y^{\text{prin}})$, ($i=1,2$), $\mZ_2$-equivariantly in the interior of $W_1^{\text{prin}}$
    and obtain a $\mathbb Z_2$-cobordism $W_2$ together with a $\mZ_2$-retract map $r_2: W_2\to Y$
    which is $3$-connected restricted to $W_2^{\text{prin}}=W_2\setminus {\rm Fix}(W_2)$, $2$-connected restricted to $\text{Fix}(W_2)$. 

   By the above construction, one can use handle cancellation (see CTC Wall's work in \cite{wall1971geometrical}), to get 
   rid of $0$ and $1$-handles in $\text{Fix}(W_2)$. In other words, there exists a Morse function  $f:\text{Fix}(W)\to [0,1]$,
   such that $f(\text{Fix}(Y))=1$, $f(\text{Fix}(\Sigma))=0$, and the critical points $\{p_1,p_2,\cdots,,p_N\}$
   of $f$ lie in the interior of $\text{Fix}(W_2)$ satisfying
   \begin{align*}
      0<f(p_1)<f(p_2)<&\cdots<f(p_N)<1;\\
       2\leq \text{coIndex}(p_1)\leq\text{coIndex}(p_2)&\leq\cdots\leq \text{coIndex}(p_N).
   \end{align*}

   Next, we extend this Morse function equivariantly to the entire $W_2$. Here, we adopt Zhou's construction \cite{10.4310/jdg/1427202766}
   with minor modifications.

    Firstly, let $\epsilon,\delta> 0 $ be sufficiently small so that $f$ has no critical point in the $2\epsilon$-neighborhood $B_{2\epsilon}(Y\cup\Sigma)$ of $Y\cup\Sigma$,
    the function $(\dist({\rm Fix}(W_2), \cdot))^2$ is smooth $\mZ_2$-invariant in $B_{2\delta}({\rm Fix}(W_2))\setminus (B_{\epsilon/2}(Y\cup\Sigma))$, and the nearest projection
    $P:B_{2\delta}({\rm Fix}(W_2))\to {\rm Fix}(W_2)$ is well defined and is a submersion in $B_{2\delta}({\rm Fix}(W_2))\setminus (B_{\epsilon/2}(Y\cup\Sigma))$. 
    Take a smooth cut-off function $\eta\in [0,1]$ on ${\rm Fix}(W_2)$ with 
    \[ \eta(x)  =\left\{
    \begin{array}{ll}
       0  & x\in {\rm Fix}(W_2)\cap \Clos(B_\epsilon(Y\cup \Sigma)) \\
       \in (0,1)  & x\in {\rm Fix}(W_2)\cap (B_{2\epsilon}(Y\cup\Sigma)\setminus \Clos(B_{\epsilon}(Y\cup\Sigma)))\\
       1 & x\in {\rm Fix}(W_2)\setminus B_{2\epsilon}(Y\cup \Sigma)
    \end{array}
    \right.
    \]
    Then by modifying $\tilde{f}(x) :=f(P(x))+ (\dist(x,\text{Fix}(W_2)))^2\cdot \eta(P(x))$ slightly in $B_\epsilon(Y\cup\Sigma)$, we obtain a smooth 
    $\mZ_2$-invariant function $\tilde{f}: B_{2\delta}({\rm Fix}(W_2)) \to [0,1]$ with no critical point in $B_{2\delta}({\rm Fix}(W_2))\cap (B_{2\epsilon}(Y\cup\Sigma))$ so that
    \begin{itemize}
        \item $\tilde{f}=1$ on $Y\cap B_{2\delta}({\rm Fix}(W_2))$, $\tilde{f}=0$ on $\Sigma\cap B_{2\delta}({\rm Fix}(W_2))$; 
        \item $x\in B_{2\delta}({\rm Fix}(W_2))$ is a critical point of $\tilde{f}$ if and only if $x\in {\rm Fix}(W_2)$ is a critical point of $f$; 
        \item ${\rm coIndex}(\tilde{f},x) = {\rm coIndex}(f,x) + 1 \geq 3$ for any critical point $x$ of $\tilde{f}$. 
    \end{itemize}
    Next, by shrinking $\epsilon>0$ even smaller, one can extend $\tilde{f}$ to a smooth $\mathbb Z_2$-invariant function on $W_2$ with no critical point in $B_\epsilon(Y\cup\Sigma)$ such that
    \begin{align*}
        \tilde{f}(Y)=1,\quad \tilde{f}(\Sigma)=0.
    \end{align*}
    Since $G$-invariant Morse function is dense in the smooth $G$-invariant functions \cite{wasserman1969equivariant} (which also has an analogue in Riemannian bands \cite[Corollary 8]{hanke2008positive}),
    there exists a $\mathbb Z_2$-invariant Morse function $\hat{f}:W_2\to[0,1]$ with no critical point in $B_{\epsilon/2}(Y\cup\Sigma)$ such that
   \begin{align*}
    \hat{f}(Y)=1,\quad \hat{f}(\Sigma)=0,\quad {\rm and}\quad \|\hat{f}-\tilde{f}\|_{C^2} \leq \tau,
   \end{align*}
    where $\tau$ is any small positive constant.

    Let $\varphi: \mathbb R_{\geq0}\to\mathbb R_{\geq 0}$ be cut-off function which is $1$ on $[0,\delta]$ and $0$ on $[2\delta,\infty]$.
    Then $\phi(x)=\varphi(\dist(x,\text{Fix}(W_2)))$ is a smooth $\mathbb Z_2$-invariant cut-off function on $W_2$.
    Define 
    \begin{align*}
        h=\phi(x)\tilde{f}+(1-\phi)\hat{f},
    \end{align*}
    $h$ is then a $\mathbb Z_2$-invariant Morse function.
    Indeed, for any $x\in W_2$ with $\text{dist}(x,\text{Fix}(W_2))\leq \delta$, $h(x)=\tilde{f}(x)$ is a $\mZ_2$-Morse function.
    Given $x\in W_2$ with $\delta\leq\text{dist}(x,\text{Fix}(W_2))\leq 2\delta$, since $|\nabla \tilde{f}|(x)$ is uniformly bounded away from $0$ for all such $x$ and 
   \begin{align*}
    |\nabla h|=|\nabla\tilde{f}+(1-\phi)(\nabla\hat{f}-\nabla\tilde{f})-\nabla\phi(\hat{f}-\tilde{f})|,
   \end{align*}
   we have $|\nabla h|(x)>0$ provided $\tau>0$ sufficiently small. 
   For any $x\in W_2$ with $\text{dist}(x,\text{Fix}(W_2))\geq 2\delta$, $h(x)=\hat{f}(x)$ is also a $\mZ_2$-Morse function. 
   Together, we conclude that $h$ is a $\mathbb Z_2$-invariant {\em special} Morse function in the sense of Definition \ref{defn:special morse}.

   Combining with the third bullet above, we are now in the position to cancel codimension $0, 1, 2$ handles in $W_2^{\text{prin}}$ equivariantly. 
   Indeed, this can be done by following Hanke's argument in the last part of the proof of \cite[Theorem 15]{hanke2008positive}, which removes the critical orbits of coindex $0,1,2$ in $W_2^{\text{prin}}$. 

   Consequently, we obtain a $\mathbb Z_2$-Morse function $\tilde{h}$ on $W_2$ with conidex of every critical orbit not less than $3$, which implies that $Y$ can be obtained from $\Sigma$ by finitely many steps of $\mZ_2$-equivariant surgeries with codimension greater or equal to $3$.
\end{proof}
\begin{remark}
    We believe Proposition \ref{prop:equivariant surgery} should still 
    hold for certain finite group actions (e.g. \cite[Corollary 16]{hanke2008positive}), and one can generalize Theorem \ref{thm:equivariant s1 stability} to these equivariant cases. 
\end{remark}

The motivation for proving Proposition \ref{prop:equivariant surgery} comes from the following generalization of the surgery result due to Gromov-Lawson and Schoen-Yau, which was stated in Bergery \cite{bergery1983scalar}, and the proof can also be found in Hanke \cite[Theorem 2]{hanke2008positive}. 

\begin{theorem}[cf. \cite{bergery1983scalar} Theorem 11.1]\label{equivariant gromov-lawson}
    Let $M$ be a (not necessarily compact) $G$-manifold equipped with a $G$-invariant metric of positive scalar curvature
    and let $N$ be obtained from $M$ by equivariant surgery of codimension at least $3$. Then $N$ carries a $G$-invariant
    metric with positive scalar curvature. 
\end{theorem}

Now we have all the ingredients we need to complete the proof of Theorem \ref{thm:equivariant s1 stability}.

\begin{proof}[Proof of Theorem \ref{thm:equivariant s1 stability}]
    Suppose $M^n$ admits a $\mathbb Z_2$-metric of positive scalar curvature, then it is immediately true that 
    $M^n\times S^1$ admits a $\mathbb Z_2$-metric of positive scalar curvature. 
    
    For the other direction, we argue by contradiction. Suppose $M^n\times S^1$ admits a $\mathbb Z_2$-metric
    $g_{M\times S^1}$ of positive scalar curvature, we assume that $M^n$ does not admit a $\mathbb Z_2$-metric of positive scalar curvature. 
    Then lift $g_{M\times S^1}$ to the covering space $(M^n\times\mathbb R, g_{M\times\R})$. For any interval 
    $(a,b)\subset\mathbb R$, the Riemannian band $(M^n\times (a,b),g_{M\times\R})$ does not have any separating $\mathbb Z_2$-hypersurface which admits 
    a $\mathbb Z_2$-metric with positive scalar curvature. 
    Otherwise, by Proposition \ref{prop:equivariant surgery} and Theorem \ref{equivariant gromov-lawson}, $M^n$ admits a 
    $\mathbb Z_2$-metric with positive scalar curvature, contradiction achieved. 
    Then we can apply Proposition \ref{prop:equivariant 2pi/n} to conclude that the width of $M^n\times(a,b)$ is bounded above 
    by a uniform constant only depending on $g_{M\times S^1}$. 
    This leads to a contradiction since $|b-a|$ can be arbitrarily large, and the width of the corresponding band goes to infinity as $|b-a|\to\infty$.

    Therefore, we have shown that 
    \begin{align*}
        \sigma_{\mathbb Z_2}(M^n)>0\iff \sigma_{\mathbb Z_2}(M^n\times S^1)>0.
    \end{align*}
\end{proof}

Finally, we prove Theorem \ref{thm:compact s1 stability}.

\begin{proof}[Proof of Theorem \ref{thm:compact s1 stability}]
    If $n=3$, this follows directly from the classification of compact $3$-manifolds, see \cite{carlotto2021constrained}.

    Now we assume $n=5,6$.
    
    Suppose $M^n$ admits a $\mathbb Z_2$-metric of positive scalar curvature in the interior 
    and the boundary is minimal. Then by taking the product metric, we know the same is true 
    for $M^n\times S^1$.

    For the other direction, we need to use Theorem \ref{thm:equivariant s1 stability}. 
    Suppose $\sigma(M^n\times S^1)>0$, then it admits a metric of positive scalar curvature in the interior, and the boundary has positive mean curvature, see \cite{escobar1992conformal}. 
    Denote by $DM^n$ the doubled manifold of $M$.
    Adopting Gromov-Lawson's construction 
    in Theorem 5.7 of \cite{gromov1980spin}, we can obtain a $\mathbb Z_2$-metric on the doubled manifold $DM^n\times S^1$ with
    positive scalar curvature, which by Theorem \ref{thm:equivariant s1 stability} implies that 
    $DM^n$ admits a $\mathbb Z_2$-metric with positive scalar curvature. 
    Restrict this metric on $M^n$, we obtain a metric on $M^n$ with positive scalar curvature in the 
    interior and a minimal boundary. Therefore, we conclude that $\sigma(M^n)>0$.
\end{proof}

\bibliographystyle{abbrv}

\providecommand{\bysame}{\leavevmode\hbox to3em{\hrulefill}\thinspace}
\providecommand{\MR}{\relax\ifhmode\unskip\space\fi MR }
\providecommand{\MRhref}[2]{%
  \href{http://www.ams.org/mathscinet-getitem?mr=#1}{#2}}
\providecommand{\ref}[2]{#2}

\bibliography{reference}   

\end{document}